\documentclass[11pt, reqno]{amsart}

\usepackage{amsmath, amssymb, amsthm}
\usepackage{geometry}
\usepackage{enumitem}
\usepackage{graphicx}
\usepackage{hyperref} % Hyperref should generally be loaded last

\newtheorem{theorem}{Theorem}
\newtheorem{lemma}[theorem]{Lemma}
\newtheorem{corollary}[theorem]{Corollary}
\theoremstyle{definition}
\newtheorem{definition}[theorem]{Definition}
\theoremstyle{remark}
\newtheorem*{remark}{Remark}

\newcommand{\R}{\mathbb{R}}
\newcommand{\Q}{\mathbb{Q}}
\newcommand{\Z}{\mathbb{Z}}

\title[Solution to a Problem of Erd\H{o}s Concerning Distances and Points]{Solution to a Problem of Erd\H{o}s Concerning Distances and Points}

\author{Benjamin Grayzel} 
\address{Dartmouth College, Departments of Mathematics \& Computer Science}
\email{benjamin.a.grayzel.gr@dartmouth.edu}

\date{\today}

\begin{document}

\begin{abstract}
In 1997, Erd\H{o}s asked \cite{erdos1997} whether for arbitrarily large $n$ there exists a set of $n$ points in $\R^2$ that determines $O(\frac{n}{\sqrt{\log n}})$ distinct distances while satisfying the local constraint that every $4$-point subset determines at least $3$ distinct pairwise distances. We use $n$-point subsets of an $m\times m$ box in the anisotropic lattice
\[
L = \{(x,\sqrt{2}y):x,y \in \Z\} \subset \R^2.
\]
The distinct-distance bound follows from Bernays' theorem applied to the binary quadratic form $u^2+2v^2$ \cite{brinkmoreeosburn2011}, following earlier work on lattices with few distances \cite{sheffer2014blog, moreeosburn2006}. The local $4$-point constraint is verified using Perucca's classification \cite{perucca4points} of $4$-point two-distance sets, together with the observation noted here that this lattice excludes the remaining regular-pentagon trapezoid configuration.
\end{abstract}

\maketitle

\section{Introduction}

Erd\H{o}s asks whether there exists a set of $n$ points in $\R^2$ such that every subset
of $4$ points determines at least $3$ distinct distances, yet the total number of distinct distances is
\[
\le C\frac{n}{\sqrt{\log n}}
\]
for some constant $C > 0$ \cite{erdos1997}, which we write as $O\!\left(\frac{n}{\sqrt{\log n}}\right)$. A formulation also appears as the 659th problem on the Erd\H{o}s problems website \cite{erdosproblems}.
 
 Constructions of planar point sets with few distinct distances arising from anisotropic lattices, together with the use of Bernays' theorem (or related results) to count represented integers, appear in earlier work of Moree and Osburn \cite{moreeosburn2006}; see also Sheffer's discussion \cite{sheffer2014blog}. The additional point in the present note is that the specific lattice $\Z\times \sqrt2\,\Z$ also excludes the remaining $4$-point two-distance configuration consisting of four vertices of a regular pentagon (equivalently, the associated isosceles trapezoid). Combined with the complete similarity classification of $4$-point two-distance sets \cite{perucca4points}, this verifies the local constraint, answering Erd\H{o}s’ question in the affirmative.

\begin{theorem}\label{thm:main}
For every integer $n\ge 2$ there is a set $P\subset \R^2$ with $|P|=n$ satisfying:
\begin{enumerate}[label=(\roman*)]
\item every $4$-point subset of $P$ determines at least $3$ distinct pairwise distances; and
\item $\bigl|\{\|p-q\|:\ p,q\in P,\ p\neq q\}\bigr| = O\!\left(\frac{n}{\sqrt{\log n}}\right)$.
\end{enumerate}
\end{theorem}

\section{The Construction}

Consider the anisotropic lattice
\[
L = \{(x,\sqrt{2}y):x,y \in \Z\} \subset \R^2.
\]
For an integer $m\ge 1$, define the $m\times m$ box subset 
\[
P_m=\{(i,\sqrt2\,j): 0\le i,j\le m-1\}.
\]
Note that $|P_m|=m^2$.

\begin{definition}
For a finite set $P\subset\R^2$, let
\[
D(P) := \{\|p-q\| : p,q\in P,\ p\neq q\}
\]
denote its set of (nonzero) distinct pairwise distances.
\end{definition}
\section{Verification of the distinct-distance bound}\label{sec:distinct-dist}

Since $d\mapsto d^2$ is a bijection on the positive reals, we may bound $|D(P_m)|$ by counting distinct \emph{squared} distances. For $p=(x_1,\sqrt2\,y_1)$ and
$q=(x_2,\sqrt2\,y_2)$ in $L$,
\[
\|p-q\|^2=(x_1-x_2)^2+2(y_1-y_2)^2 = u^2+2v^2
\]
where $u,v\in\Z$. For $p,q\in P_m$ we have $|u|\le m-1$ and $|v|\le m-1$, hence
\[
\|p-q\|^2 \le (m-1)^2+2(m-1)^2 < 3m^2.
\]

Distinct squared distances in $P_m$ are exactly the values of the quadratic form
$Q(u,v)=u^2+2v^2$ with $|u|,|v|\le m-1$, and as such are bounded by $3m^2$. Therefore
\[
|D(P_m)| \le B_Q(3m^2),
\]
where $B_Q(x)$ denotes the number of positive integers $\le x$ represented by $Q$.

\begin{theorem}[Bernays {\cite[Eq.~(2)]{brinkmoreeosburn2011}}]\label{thm:bernays}
Let $f(X,Y)=aX^2+bXY+cY^2$ be a primitive integral binary quadratic form with non-square
discriminant $\Delta=b^2-4ac$, and assume $f$ is positive definite. Let $B_f(x)$ denote
the number of positive integers $\le x$ represented by $f$. Then there exists a constant $C_\Delta>0$
(depending on $\Delta$) such that
\[
B_f(x)\sim C_\Delta\,\frac{x}{\sqrt{\log x}}
\qquad (x\to\infty).
\]
\end{theorem}

The form $Q(u,v)=u^2+2v^2$ is primitive and positive definite, with discriminant
\[
\Delta = 0^2 - (4)\cdot(1)\cdot (2) = -8,
\]
which is not a square. Theorem~\ref{thm:bernays} therefore applies to $Q$, giving
\[
B_Q(x) \sim C_{-8}\,\frac{x}{\sqrt{\log x}}
\qquad (x\to\infty).
\]
In particular,
\[
|D(P_m)| \le B_Q(3m^2)
= O\!\left(\frac{m^2}{\sqrt{\log m}}\right)
= O\!\left(\frac{|P_m|}{\sqrt{\log |P_m|}}\right).
\]

\begin{corollary}\label{cor:general-n}
For every integer $n\ge 2$, there exists a set $P\subset\R^2$ with $|P|=n$ such that
\[
|D(P)| = O\!\left(\frac{n}{\sqrt{\log n}}\right).
\]
\end{corollary}

\begin{proof}
Let $m=\lceil\sqrt{n}\rceil$ and choose any $n$-point subset $P\subset P_m$.
Removing points cannot increase the number of distinct distances, so $|D(P)|\le |D(P_m)|$.
Also $m^2\asymp n$ and $\log m \asymp \log n$, hence
\[
|D(P)| \le |D(P_m)|
= O\!\left(\frac{m^2}{\sqrt{\log m}}\right)
= O\!\left(\frac{n}{\sqrt{\log n}}\right).
\]
\end{proof}

\section{Verification of the local $4$-point constraint}\label{sec:local4}

The local constraint fails if and only if the 4-point set determines \emph{exactly two} distinct distances. A complete classification of $4$-point two-distance configurations (up to similarity)
is given in \cite{perucca4points} and depicted in Figure~\ref{fig:forbidden_shapes}. Note that the only similarity types that do \emph{not} contain an equilateral triangle are the square and the regular-pentagon trapezoid.

\begin{figure}[ht]
  \centering
  \includegraphics[width=\linewidth]{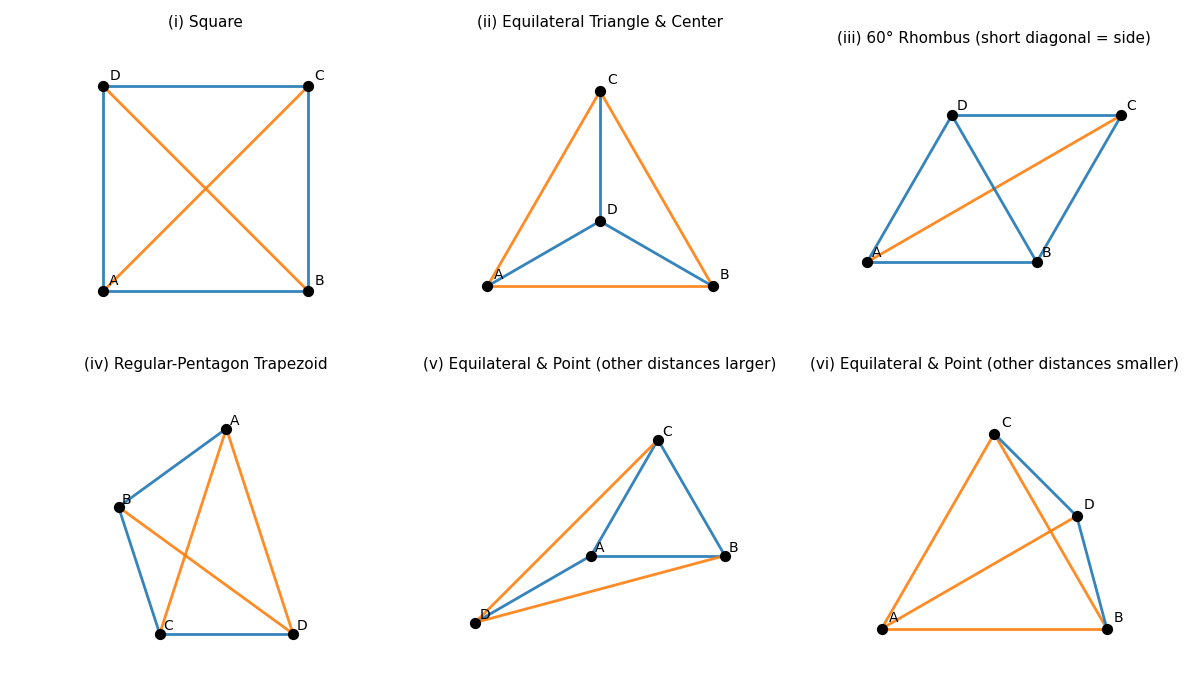}
  \caption{The six $4$-point two-distance configurations (up to similarity). Short vs.\ long segments are color-coded.}
  \label{fig:forbidden_shapes}
\end{figure}

\begin{theorem}\label{thm:local4}
For every $m\ge 1$, every $4$-point subset of $P_m$ determines at least $3$ distinct pairwise distances.
\end{theorem}

\begin{proof}
Let $S\subset P_m$ with $|S|=4$. Suppose that $S$ determines exactly two distinct distances.
By Perucca's classification \cite{perucca4points}, either $S$ is a square, $S$ contains an equilateral triangle, or $S$ is
similar to the regular-pentagon trapezoid.

The square case is impossible by Lemma~\ref{lem:nosquare}, and the equilateral-triangle case is impossible by
Lemma~\ref{lem:noequilateral}. The regular-pentagon trapezoid case is impossible by Lemma~\ref{lem:nopentagon}.
Hence $S$ cannot determine only two distances and must determine at least three.
\end{proof}

\subsection{Squares}

\begin{lemma}\label{lem:nosquare}
The lattice $L$ contains no non-degenerate square.
\end{lemma}

\begin{proof}
If four points of $L$ form a square, then the vector between two adjacent vertices is a nonzero vector
\[
 \vec v=(u,\sqrt2\,v)\in L,\qquad (u,v)\in\Z^2\setminus\{(0,0)\},
\]
and the other side vector is obtained by a $90^\circ$ rotation:
\[
 \begin{pmatrix} 0 & -1\\[2pt] 1 & 0\end{pmatrix} \vec v = (-\sqrt2\,v,\,u).
\]
As such, $(-\sqrt2\,v,\,u)$ must belong to L, which implies $-\sqrt2\,v \in \Z$ and $u\in\sqrt{2}\Z $. This is only possible if $ \vec v = 0$ 
and thus forces a contradiction.
\end{proof}

\subsection{Equilateral triangles}

\begin{lemma}\label{lem:noequilateral}
The lattice $L$ contains no non-degenerate equilateral triangle.
\end{lemma}

\begin{proof}
If $p,q,r\in L$ form an equilateral triangle, then $\vec v=q-p\in L$ is nonzero and $r-p$ equals $R_{\pm 60}\vec v$, where
\[
R_{60}=\begin{pmatrix} \tfrac12 & -\tfrac{\sqrt3}{2}\\[2pt] \tfrac{\sqrt3}{2} & \tfrac12\end{pmatrix}.
\]
Write $\vec v=(u,\sqrt2\,v)$ with $u,v\in\Z$, not both $0$. Then
\[
R_{60}\vec v=\left(\frac{u-\sqrt6\,v}{2},\ \frac{\sqrt3\,u+\sqrt2\,v}{2}\right).
\]
For $R_{60}\vec v$ to lie in $L$, there must exist $k\in\Z$ such that
\[
\frac{\sqrt3\,u+\sqrt2\,v}{2}=\sqrt2\,k,
\]
so $\sqrt3\,u=\sqrt2\,(2k-v)\in\Q(\sqrt2)$. Since $\sqrt3\notin\Q(\sqrt2)$, this forces $u=0$. Then the first coordinate
becomes $-\frac{\sqrt6\,v}{2}$, which is irrational unless $v=0$. This implies $u=v=0$, forcing a contradiction with $\vec v\ne 0$.
The same argument holds for $R_{-60}$.
\end{proof}

\begin{remark}
The exclusion of squares and equilateral triangles in rectangular or anisotropic lattice constructions is standard; see \cite{sheffer2014blog} for related discussion.
\end{remark}

\subsection{The regular-pentagon trapezoid}

\begin{lemma}\label{lem:nopentagon}
The lattice $L$ contains no $4$-point set similar to the isosceles trapezoid arising from four vertices of a regular pentagon.
\end{lemma}

\begin{proof}
In this configuration, the ratio of the diagonal length to the side length equals the golden ratio
$\varphi=\frac{1+\sqrt5}{2}$, hence the ratio of squared lengths (diagonal to side) is $\varphi^2=\frac{3+\sqrt5}{2}$.

For points in $L$, every squared distance equals $u^2+2v^2\in\Z$. Therefore any ratio of (nonzero) squared distances
determined by a $4$-point subset of $L$ is rational, and in particular cannot equal $\varphi^2$.
\end{proof}

\section{Conclusion}

We have shown that the anisotropic lattice box $P_m \subset \Z \times \sqrt2\,\Z$ satisfies both the asymptotic
distinct-distance bound
\[
|D(P_m)| = O(\frac{|P_m|}{\sqrt{\log |P_m|}})
\]
and the local constraint that every $4$-point subset determines at least three distinct pairwise distances. For a general $n$,
taking any $n$-point subset of $P_{\lceil\sqrt n\rceil}$ preserves the local constraint and gives the same asymptotic
upper bound (up to constants), as in Corollary~\ref{cor:general-n}. The construction ultimately works because $60^\circ$ and $90^\circ$ rotations, as well as the golden-ratio distance relation, are incompatible with the quadratic field
$\Q(\sqrt2)$ and this last obstruction rules out the regular-pentagon trapezoid configuration.

\section*{Acknowledgements}
Gemini 3.0 provided the core ideation behind this approach, and specifically surfaced the observation that the same lattice excludes the regular-pentagon trapezoid configuration. The author relied extensively on GPT-5.2 and Gemini 3.0 for drafting, editing, and exploring possible proof strategies. All mathematical claims, proofs, and citations were independently verified by the author, who takes responsibility for any errors. I am grateful to Professor Peter Doyle and Professor Peter Winkler for helpful discussions and feedback on early drafts. I also thank the participants in the public comment thread on the Erd\H{o}s Problems website \cite{erdosproblems}, particularly Desmond Weisenberg for identifying the local constraint as six forbidden shapes, Terrence Tao for highlighting prior work from Sheffer and Moree--Osburn, and Boris Alexeev for formalizing the proof in Lean.

\end{document}